\documentclass[10pt]{amsart}

\usepackage{amssymb,array}
\usepackage{amsmath}
\usepackage{amsthm}
\usepackage{amsbsy,mathrsfs}
\usepackage{bm}
\usepackage[english]{babel}
\usepackage{graphicx}
\usepackage{color}
\usepackage{hyperref}
\usepackage[toc,page]{appendix}
\usepackage{thmtools}
\usepackage{thm-restate}
\usepackage[margin=1.2in]{geometry}
\usepackage{cleveref}
\usepackage{subfig}
\usepackage{subcaption}
\usepackage{float}
\usepackage{url}

\theoremstyle{plain}
\newtheorem{theorem}{Theorem}[section]
\newtheorem*{theorem*}{Theorem}
\newtheorem{lemma}[theorem]{Lemma}
\newtheorem{ex}[theorem]{Example}

\newtheorem{corollary}[theorem]{Corollary}
\newtheorem{proposition}[theorem]{Proposition}
\newtheorem{computational theorem}[theorem]{Computational Theorem}

\theoremstyle{definition}

\newtheorem{remark}[theorem]{Remark}

\newcommand{\R}{\mathbb{R}}

\renewcommand{\H}{{\mathbb{H}}}

\newcommand{\M}{\mathcal{M}}
\newcommand{\T}{\mathcal{T}}

\DeclareMathOperator{\tr}{\operatorname{tr}}
\DeclareMathOperator{\Hyp}{\operatorname{Hyp}}

\theoremstyle{plain} 
\newcommand{\thistheoremname}{}
\newtheorem*{genericthm*}{\thistheoremname}

\begin{document}
\title[Poisson bracket of trace functions]{Poisson bracket of trace functions on the Fuchsian locus and Wolpert's formulas}

 \author{Deblina Das}
 \address{Department of Mathematics, Indian Institute of Technology Palakkad}
 \email{212114002@smail.iitpkd.ac.in}
 \author{Arpan Kabiraj}
 \address{Department of Mathematics, Indian Institute of Technology Palakkad}
\email{arpaninto@iitpkd.ac.in}

\begin{abstract}

We develop a systematic method for computing traces of products of M\"obius transformations associated with oriented geodesics on a hyperbolic surface. The method is based on a normalization of matrices in ${SL}_2(\mathbb R)$ which expresses trace identities in terms of hyperbolic lengths, intersection angles, and signed distances along geodesics. Using these trace computations together with Goldman's description of the Atiyah-Bott-Goldman symplectic form on the character variety, we derive explicit geometric formulas for the Poisson brackets of trace functions associated with closed geodesics. More precisely, we express the Poisson bracket of two trace functions and the iterated Poisson bracket of three trace functions in terms of the hyperbolic lengths of the corresponding geodesics, their intersection angles, and the signed distances between intersection points. The results naturally lead to a unified perspective for revisiting Wolpert's cosine and sine formulas and deriving new proofs of them. 
\end{abstract}

\maketitle

\section{Introduction}

Let $\Sigma$ be a closed connected oriented surface which admits a complete hyperbolic metric. We refer to such a surface as a \emph{hyperbolic surface}, and denote by $\T$ the associated Teichm\"uller space.

A fundamental result of Wolpert \cite[Theorem 3.3]{wolpert1983symplectic} establishes a deep connection between the Weil-Petersson symplectic structure on $\T$ and the geometry of closed geodesics on $\Sigma$. Given two closed curves on $\Sigma$, Wolpert showed that the Poisson bracket of their geodesic length functions is expressed as a sum over their intersection points, where each term involves the cosine of the angle between the corresponding geodesic representatives. This result is now known as \emph{Wolpert's cosine formula}. He also discovered a formula for the iterated Poisson bracket of length functions of three geodesics known as \emph{Wolpert's sine formula} (\cite[Theorem 3.4]{wolpert1983symplectic}).

 Goldman \cite{MR762512, goldman_invariant_1986} introduced a far-reaching generalization of this picture. For a broad class of Lie groups $G$, he endowed the $G$-character variety of $\pi_1(\Sigma)$ with a natural Poisson structure called the \emph{Atiyah-Bott-Goldman} Poisson structure. Given a conjugation invariant function on $G$, each free homotopy class of closed curves on $\Sigma$ determines a conjugation invariant function on the character variety. For some special classes of classical linear Lie groups $G$ and trace function on $G$, Goldman gave an explicit description of the  Poisson bracket of such functions. 
 He further showed that, for $G = {GL}_n(\mathbb{R})$ and ${SL}_2(\mathbb{R})$, the Poisson bracket admits a purely topological description.

Subsequent developments have significantly deepened this picture. Turaev \cite{turaev1991skein} introduced a Lie cobracket on curves, leading to a Lie bialgebra structure closely related to skein algebras (see also \cite{hoste1990homotopy}, \cite{MR2041630}). From a string-topological perspective, Chas and Sullivan \cite{chas1999stringtopology} interpreted the Goldman bracket as a manifestation of loop product structures. Kawazumi and Kuno \cite{MR3497295}, \cite{MR3449595}, \cite{MR3283405} developed a homological version of the Goldman Lie algebra and studied its action on the fundamental group and its completions, providing a powerful bridge between topology, geometry, and representation theory.

 These viewpoints suggest that Wolpert’s formulas may be interpreted as geometric manifestations of the Goldman bracket, realized through suitable representation-theoretic frameworks. In particular, when $G = {PSL}_2(\mathbb{R})$, the character variety is canonically related to Teichm\"uller space, and trace functions encode hyperbolic length functions. Thus, the Weil-Petersson Poisson bracket may be interpreted through the Goldman bracket via trace identities. Indeed, Goldman used this analog to obtain a proof of Wolpert's cosine formula \cite[Section 3.9]{goldman_invariant_1986}. 

The goal of this paper is to investigate this viewpoint in detail. More precisely, we develop a systematic method to compute trace identities associated with products of suitably normalized M\"{o}bius transformations arising from oriented geodesics on a hyperbolic surface. We express these traces explicitly in terms of geometric data: the translation lengths of the corresponding hyperbolic elements, the intersection angles of their axes, and the signed distances between relevant intersection points.
We summarize the main results of our paper in the following theorems. For the definitions of the terms appearing in the theorems, see Section \ref{sec:basic}.

\begin{theorem*}[\ref{thm:cosine trace}]
    Suppose $x$ and $y$ are two intersecting geodesics  on $\Sigma$ with $\ell_x=a$ and $\ell_y=b$. For the Atiyah-Bott-Goldman poisson bracket on the Fuchsian locus of $SL_2(\R)$ character variety we have $$\{\tr_x,\tr_y\}=  2\sinh\frac{a}{2}\sinh \frac{b}{2}\sum_{p\in x\cap y}  \cos\theta_p.$$
\end{theorem*}

\begin{theorem*}[\ref{thm:sine trace}]
    Suppose $x,y$ and $z$ are three closed geodesics on $\Sigma$ with $\ell_x=a$, $\ell_y=b$ and $\ell_z=c$. Denote $C_u=\cosh\frac u2$ and $S_u=\sinh\frac u2.$  For the Atiyah-Bott-Goldman Poisson bracket on the Fuchsian locus of $SL_2(\R)$ character variety we have
    \begin{align*}
        &\{\tr_z, \{\tr_x, \tr_y\}\}\\ = & \ -2C_aS_b S_c\sum_{p\in x\cap y,q\in z\cap x}\cos\theta_p\cos\theta_q-  2 S_aC_bS_c\sum_{r\in x\cap y,s\in z\cap y}\cos\theta_r\cos\theta_s\\
 &- 2S_aS_bS_c\sum_{p\in x\cap y,q\in z\cap x}\frac{\cosh\big(\ell_{pq}-\frac{a}{2}\big)}{S_a}\sin\theta_p\sin\theta_q
 + 2S_aS_bS_c\sum_{r\in x\cap y,s\in z\cap y}\frac{\cosh\big(\ell_{rs}-\frac{b}{2}\big)}{S_b}\sin\theta_r\sin\theta_s
    \end{align*}
  where $\ell_{pq}$ (respectively $\ell_{rs}$) denotes the distance from $p$ to $q$ (respectively $r$ to $s$) along $x$ (respectively along $y$). 
\end{theorem*} 
The Atiyah-Bott-Goldman symplectic form is a scalar multiple of the Weil-Petersson symplectic form \cite{MR762512}. Therefore, the above results naturally lead to a unified perspective for revisiting Wolpert's cosine and sine formulas and deriving new proofs of them. 

\begin{theorem*}[\ref{thm:cosine form} Wolpert's cosine formula]  Suppose $x$ and $y$ are two closed geodesics on $\Sigma$. For the Atiyah-Bott-Goldman Poisson bracket on the Teichm\"{u}ller space  we have
$$\{\ell_x,\ell_y\}=2\sum\limits_{p\in x\cap y}\cos{\theta}_p.$$
\end{theorem*}

\begin{theorem*}[\ref{thm:sine form} Wolpert's sine formula]
   Suppose $x,y$ and $z$ are three closed geodesics on $\Sigma$ with $\ell_x=a$, $\ell_y=b$ and $\ell_z=c$. For the Atiyah-Bott-Goldman poisson bracket on the Teichm\"{u}ller space  we have $$\{\ell_z,\{\ell_x,\ell_y\}\}=- 2\sum_{p\in x\cap y,q\in z\cap x}\frac{\cosh\big(\ell_{pq}-\frac{a}{2}\big)}{S_a}\sin\theta_p\sin\theta_q
 + 2\sum_{r\in x\cap y,s\in z\cap y}\frac{\cosh\big(\ell_{rs}-\frac{b}{2}\big)}{S_b}\sin\theta_r\sin\theta_s.$$
\end{theorem*}

For alternative approaches to Wolpert’s formulas, together with related results and applications, we refer the reader to \cite{MR3359033}, \cite{Gendulphe2015}, \cite{MR615620}, \cite{MR690845}, \cite{MR1810217}, \cite{MR4911761}, \cite{ablondi2026rhamperspectivesymplecticgeometry}, \cite{MR4149828}, \cite{MR4797740}, \cite{MR4108617}, \cite{MR4959391}, \cite{andersen2021kontsevichgeometrycombinatorialteichmuller}.

\subsection{Organization of the paper} In section \ref{sec:basic} we recall the basic preliminaries. In section \ref{sec:trace id} we discuss the trace of the product of two and three isometries and develop essential lemmas required to prove our main theorems. Finally in Section \ref{sec:Poisson bracket of trace and length functions} we prove Theorem \ref{thm:sine trace}, \ref{thm:cosine trace}, \ref{thm:sine form}, \ref{thm:cosine form}.

\section{Acknowledgment} The authors would like to thank Prof. Nariya Kawazumi for carefully reading our manuscript and for several helpful comments and suggestions.

\section{Preliminaries}\label{sec:basic}
Let $\Sigma$ be a closed connected oriented surface with a complete hyperbolic metric. We call such a surface a \emph{hyperbolic surface}. The universal cover of $\Sigma$ is the upper half plane $\mathbb H$ with the Poinca\'{r}e metric and we fix the orientation of $\H$ to be anti-clockwise orientation. The isometry group of $\H$ is naturally identified with the group $PSL_2(\mathbb R)$ where the elements act as M\"{o}bius transformations. We also consider matrices of $SL_2(\mathbb R)$ acting on $\H$ by M\"{o}bius transformations. Trace of a matrix $A$ is denoted by $\tr(A)$.  We call a matrix $A\in SL_2(\mathbb R)$ \emph{hyperbolic} if $|\tr (A)|>2$ and denote the set of all hyperbolic matrices by  $\Hyp=\{A\in SL_2(\mathbb R): |\tr A|>2\}.$ For any $A\in SL_2(\mathbb R)$ corresponding to a hyperbolic isometry, the geodesic axis of the corresponding M\"{o}bius transformation is denoted by $T_A$. 

We denote the fundamental group $\pi_1(\Sigma)$ of $\Sigma$ simply by $\pi_1$. There is a one-to-one correspondence between the conjugacy classes of elements of   $\pi_1$ and the free homotopy classes of oriented closed curves on $\Sigma$. We use this correspondence throughout the paper without explicitly mentioning it, and denote both sets by $\pi$. Each free homotopy class of a closed oriented curve $x$ has a unique geodesic in its free homotopy class whose length is denoted by $\ell_x$. Denote by $\T$ the Teichm{\"u}ller space associated with the surface $\Sigma$. 

Given any reductive linear Lie group $G$, let $M_G=Hom(\pi_1,G)/G$ be the $G$-character variety where the quotient is by conjugation. In \cite{MR762512} Goldman proved that $M_G$ admits a symplectic from called the \emph{Atiyah-Bott-Goldman} symplectic form  which is a generalization of the classical  Atiyah-Bott and Weil-Petersson symplectic structures.  For our discussion, we consider two deformation spaces:\begin{itemize}
    \item $\M_W=Hom(\pi_1,PSL_2(\mathbb R))/PSL_2(\mathbb R)$
    \item $\M=Hom(\pi_1,SL_2(\mathbb R))/SL_2(\mathbb R)$
\end{itemize} 

The Teichm{\"u}ller space $\T$ is naturally identified as a subset of $\M_W$. 
For each element $\Phi\in \M_W$ we choose a lift $\tilde{\Phi}\in \M$ which varies continuously on $\T$. We identify $\T$ as a subset of $\M$ using this lift (see \cite{abikoff1983lifting} for details).

Given any $x\in \pi_1$, we have natural trace maps $\tr_x$ from $\M$ to $\mathbb R$ defined by $\tr_x(\Phi)=\text{trace} (\Phi(x))$. For $\M_W$ we have the $|\tr_x|$ map defined by the same function. Because trace is conjugacy invariant, $\tr_x$ depends only on the conjugacy class of $x$.  On the Teichm{\"u}ller space $\T$, we have the corresponding length function $\ell_x$ defined by $\ell_x(\Phi)=$ length of the unique geodesic (with respect to $\Phi$) in the free homotopy class of $x$. 

\begin{remark}
In this paper we shall work primarily on the \emph{Fuchsian locus} of $\M$. By this we mean the representations which arise as lifts to $SL_2(\R)$ of discrete and faithful Teichm\"uller representations into $PSL_2(\R)$. Unless otherwise stated, all representations considered in $\M$ are assumed to lie in this Fuchsian locus. For details, see \cite{MR952283}.
\end{remark}
 Consider $\Phi\in \T$. Given any $x\in \pi_1$ consider the isometry corresponding to $\Phi(x)$. There are two possible choices for this isometry as an element of  $SL_2(\R)$. For a chosen representative $A\in SL_2(\R)$ for this isometry, let $\varepsilon_x$ be the sign of $\tr(A)$.    
 With the above identification, we have the following relation in $\T$ as a subset of $\M_W$ and $\M$ 
$$|\tr_x(\Phi)|=\varepsilon_x\tr_x=\varepsilon_x\tr(A)\,\,\,\text{and}\,\,\,\tr_x=2\varepsilon_x \cosh\frac{\ell_x}{2}$$
where $\varepsilon_x=\pm 1$ depends only on $x$.

Given any two geodesics $x$ and $y$ (on $\Sigma$ or on $\H$) and an intersection point $p$ between them, the sign of intersection at $p$, $\epsilon_p(x,y)$ is defined to be $1$ if $(x'(p), y'(p))$ agrees with the orientation at $p$ and $-1$ otherwise. Observe that $\epsilon_p(x,y)=-\epsilon_p(y,x)$ and $\epsilon_p(x,y)=-\epsilon_p(x^{-1},y)$

In \cite{goldman_invariant_1986}, Goldman showed that the space of all trace functions on $M_G$ admits a Poisson algebra structure from the Atiyah-Bott-Goldman symplectic form. We denote all these Poisson brackets by $\{ \,,\}$ where the space should be clear from the context. Goldman also gave an explicit description of these Poisson brackets for many classes of Lie groups. We recall the results relevant to our discussion below. 

\begin{theorem}\cite[Theorem 3.13 and 3.14]{goldman_invariant_1986}\label{thm:Goldman}
    Let $x$ and $y$ be two oriented closed curves on $\Sigma$ intersecting each other transversally in double points.
    For $\M$ we have  $$\{\tr_x,\tr_y\}=\frac 12\sum_{p\in x\cap y}\epsilon_p(x,y) (\tr_{x*_py}-\tr_{x*_py^{-1}}).$$
 Where $x*_py$ denotes the loop product of $x$ and $y$ at $p$ and $y^{-1}$ denotes the curve $y$ with opposite orientation.
\end{theorem}

\begin{remark}
    The space $M_G$ is generally a singular variety and not a smooth manifold. However, it admits a dense subset which is a smooth manifold. To avoid technicalities, for our discussion, we pretend that $M_G$ is a smooth manifold. See \cite{MR827264} for details. 
\end{remark}


\section{Trace identities}\label{sec:trace id}

Let $A\in SL(2,\mathbb R)$ be a hyperbolic isometry with translation length $a>0$. Denote $C_u=\cosh\frac u2$ and $
S_u=\sinh\frac u2.$ Therefore  $ |{\tr}(A)|=2C_a$ and   $C_a= \varepsilon_A  \frac 12 \tr(A)$ where $\varepsilon_A\in \{1,-1\}$ determines the sign of the trace of $A$.  Define $ N_X:=\displaystyle\frac{\varepsilon_A A- C_aI}{S_a}$ or
equivalently,  $A:=\varepsilon_A (C_a I+S_a N_X).$ With this convention, we have the following lemma.
 \begin{lemma}\label{lem: rep hyp elmt}
 $A$ can be written as $ A=\varepsilon_A (C_a I+S_a N_X)$ where 
${\tr}(N_X)=0,$
 and $N_X^2=I.$ Moreover, $ A^{-1}=\varepsilon_A(C_a I-S_a N_X).$
 \end{lemma}
\begin{proof}
By the Cayley-Hamilton theorem,
$A^2-{\tr}(A)A+I=0. $
Since $\varepsilon_A{\tr}(A)=2C_a,$ we get
$A^2-2\varepsilon_A C_a A+I=0.$
Then $$(\varepsilon_A A- C_aI)^2=
 A^2-2\varepsilon_AC_aA+C_a^2I = (C_a^2-1)I=S_a^2I.
$$
Therefore
$$N_X^2 =
\left(\frac{\varepsilon_A A- C_aI}{S_a}\right)^2=I.$$
Moreover
$${\tr}(N_X)=\frac{\varepsilon_A{\tr}( A)-2C_a}{S_a}=0.$$
Also $A^2-2\varepsilon_A C_aA+I=0$ implies $A^{-1}=2\varepsilon_AC_aI-A=2\varepsilon_AC_aI-(\varepsilon_A(C_AI+S_aN_X))=\varepsilon_A (C_aI-S_aN_X).$
\end{proof}
 Geometrically, $N_X$ records the oriented axis $X$ of the hyperbolic
element $A$. Replacing $A$ by $A^{-1}$ changes $N_X$ to $-N_X$, which corresponds to reversing the orientation of the axis.
\begin{remark}
    The decomposition mentioned in the above lemma is due to Goldman (see \cite[Section 4.2]{MR2497777}). 
\end{remark}
\begin{ex}\label{ex:stdhyp} Consider the standard diagonal hyperbolic translation
$ A=
\begin{pmatrix}
e^{a/2}&0\\
0&e^{-a/2}
\end{pmatrix}.
$ Then $$ A=C_a
\begin{pmatrix}
1&0\\
0&1
\end{pmatrix}
+
S_a
\begin{pmatrix}
1&0\\
0&-1
\end{pmatrix}.
$$
Thus, in this case $ N_X=
\begin{pmatrix}
1&0\\
0&-1
\end{pmatrix},
$ which corresponds to the oriented $Y$-axis in the upper half-plane
model.
\end{ex}
Let $x,y$ and $z$ be three closed geodesics on $\Sigma$. Let $A,B,C\in SL(2,\mathbb R)$ be three hyperbolic isometries which are lifts of the deck transformations corresponding to $x,y$ and $z$ respectively. Suppose the positive translation lengths of $A,B$ and $C$ are  $a,b$ and $c$ respectively. 

By Lemma \ref{lem: rep hyp elmt} we  write $$
A=\varepsilon_A(C_aI+S_aN_X), \qquad B=\varepsilon_B(C_bI+S_bN_Y), \qquad C=\varepsilon_C(C_cI+S_cN_Z).$$ 
Then we have the following lemma.
\begin{lemma}\label{lem:gen tr 
ident}
$$
\begin{aligned}
   \frac12{\tr}(AB)=&\ \varepsilon_A\varepsilon_B(C_aC_b+S_aS_b\frac12{\tr}(N_XN_Y)), \\
   \frac12{\tr}(BC)=&\ \varepsilon_B\varepsilon_C(C_bC_c+S_bS_c\frac12{\tr}(N_YN_Z)),\\
   \frac12{\tr}(AC)=&\ \varepsilon_A\varepsilon_C (C_aC_c+S_aS_c\frac12{\tr}(N_XN_Z)),\\
    \frac12{\tr}(ABC)=\varepsilon_A\varepsilon_B\varepsilon_C (C_aC_bC_c
+&
S_aS_bC_c\frac12\tr(N_XN_Y)
+
S_aC_bS_c\frac12\tr(N_XN_Z) \\
+&
C_aS_bS_c
\frac12\tr(N_YN_Z) 
+
S_aS_bS_c
\frac12\tr(N_XN_YN_Z)).
\end{aligned}
$$
\end{lemma}
\begin{proof}
With the above notations,    $$AB= \varepsilon_A\varepsilon_B(C_aC_bI+C_aS_b N_Y+S_aC_bN_X+S_aS_bN_XN_Y).$$
    Since $\tr (N_X)$ and $\tr(N_Y)$ both are zero,
    $$ \frac12{\tr}(AB)=\varepsilon_A\varepsilon_B(C_aC_b+S_aS_b\frac12{\tr}(N_XN_Y)).$$   The remaining identities follow by analogous computations.
\end{proof}
\subsection{Trace of product of two isometries}

Let $A,B\in SL_2(\mathbb R)$ be two hyperbolic isometries with translation lengths $a>0$ and $b>0$ respectively such that their respective axes $T_A$ and $T_B$ intersect at $p$ at an angle $\theta_p$, where the angle is measured anticlockwise from $T_A$ to $T_B$. We have the following identity.

\begin{lemma}\label{lem:double}
$$
\begin{aligned}
\frac12{\tr}(AB)=
\varepsilon_A\varepsilon_B(C_aC_b+\epsilon_p(T_A,T_B)S_aS_b\cos\theta_p)\end{aligned}.
$$
 
\end{lemma}

\begin{proof}
 Write
$A=\varepsilon_A (C_a I+S_a N_X)$ and 
$B=\varepsilon_B (C_b I+S_b N_Y)$. Then by Lemma \ref{lem:gen tr ident} $$\frac12{\tr}(AB)=
\varepsilon_A\varepsilon_B (C_aC_b+S_aS_b
\frac12{\tr}(N_XN_Y)). $$
We have to show that $$
\frac12{\tr}(N_XN_Y)=\epsilon_p(T_A,T_B)\cos\theta_p.
$$
Because trace is invariant under conjugation, without  loss of generality, we may consider $p$
to be $i\in\mathbb H$, and we may assume that the axis of $A$ is the $Y$-axis with positive orientation. Then by Example \ref{ex:stdhyp}, 
$$ A=
\begin{pmatrix}
e^{a/2} & 0\\
0 & e^{-a/2}
\end{pmatrix} \quad 
\quad 
N_X=
\begin{pmatrix}
1&0\\
0&-1
\end{pmatrix}.
$$
Now suppose the axis of $B$ is the geodesic passing through $i$ making oriented angle $\theta_p$ with
$Y$-axis. Then it is obtained by rotating the axis of $A$ anticlockwise through an angle $\theta_p$ at $i$. The rotation matrix  by angle $\theta_p$ is
$$
K_{\theta_p}=
\begin{pmatrix}
\cos(\theta_p/2)& \sin(\theta_p/2)\\
-\sin(\theta_p/2)&\cos(\theta_p/2)
\end{pmatrix}.
$$
The directions assigned to the axes $T_A$ and $T_B$ are shown on the left-hand side of  Figure \ref{fig:2 ints geod}.
\begin{figure}[h]
    \centering
    \includegraphics[width=7cm]{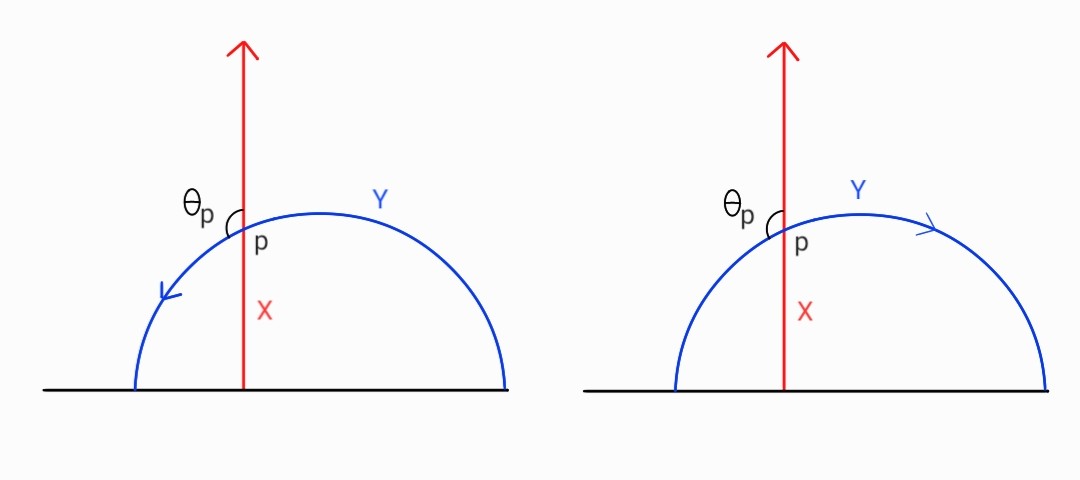}
   \caption{$T_A$ intersects $T_B$}
   \label{fig:2 ints geod}
\end{figure}
\noindent
Hence,
$
N_Y=K_{\theta_p} N_X K_{\theta_p}^{-1}.
$
A direct computation gives
$$
N_Y=
\begin{pmatrix}
\cos\theta_p & -\sin\theta_p \\
-\sin\theta_p &-\cos\theta_p
\end{pmatrix}.
$$
Therefore
$$
N_XN_Y
=
\begin{pmatrix}
1&0\\
0&-1
\end{pmatrix}
\begin{pmatrix}
\cos\theta_p&-\sin\theta_p\\
-\sin\theta_p&-\cos\theta_p
\end{pmatrix}
=
\begin{pmatrix}
\cos\theta_p&-\sin\theta_p\\
\sin\theta_p&\cos\theta_p
\end{pmatrix}
$$
and  
$
{\tr}(N_XN_Y)=2\cos\theta_p.
$
If the direction of $T_B$ is chosen as in the right-hand side of Figure \ref{fig:2 ints geod}
then we replace  $N_Y$ to $-N_Y$ and a similar computation shows ${\tr}(N_XN_Y)=-2\cos\theta_p.$ The identity follows. 
\end{proof}

\begin{corollary}\label{cor:double}
If $A,B\in SL_2(\mathbb R)$ be two hyperbolic isometries with translation lengths $a>0$ and $b>0$ respectively such that $T_A$ and $T_B$ intersects at $p$ at an angle $\theta_p$, measured anticlockwise from $T_A$ to $T_B$ then   $$\epsilon_p(T_A,T_B)(\tr(AB)-\tr(AB^{-1}))=4\varepsilon_A\varepsilon_BS_aS_b\cdot \cos\theta_p.$$
\end{corollary}
\begin{proof}
    Recall that, $\epsilon_p(T_A,T_B^{-1})=-\epsilon_p(T_A,T_B)$. The conclusion follows from the proof of Lemma \ref{lem:double}.
\end{proof}
\begin{remark}\label{rem:sign}
Observe that $\varepsilon_{AB}=\varepsilon_A\varepsilon_B$. Indeed  by Lemma \ref{lem: rep hyp elmt}, $AB$ can be written as 
\[
AB=\varepsilon_{AB}(C_\lambda I+S_\lambda N_{AB})
\]
where $\lambda$ be translation length of $AB$, so that $\tr(AB)=2\varepsilon_{AB}C_\lambda.$  By Lemma \ref{lem:gen tr 
ident}, 
\[
\tr(AB)=2\varepsilon_A\varepsilon_B(C_aC_b+S_aS_b\frac 12{\tr}(N_XN_Y)).
\]
Therefore, we obtain
$\varepsilon_{AB}C_\lambda=\varepsilon_A\varepsilon_B(C_aC_b+S_aS_b\frac 12{\tr}(N_XN_Y)).$ Since $-1\leq \frac 12\tr(N_XN_Y)\leq 1$, by Lemma \ref{lem:double} , together with the facts $C_\lambda>0$ and $C_aC_b-S_aS_b>0$, the claim follows.
\end{remark}
\subsection{Trace of product of three isometries}  
Let $A,B, C\in SL_2(\mathbb R)$ be three hyperbolic isometries with respective translation lengths
$a, b, c>0$ and axes $T_A, T_B, $ and $T_C$.
\subsubsection{Case I: $T_A$ intersects $T_C$ }\label{subs:case1}
Suppose, the axes $T_A$ and $T_B$ intersect at $p$ at an angle $\theta_p$, measured anticlockwise from $T_A$ to $T_B$; the axes $T_A$ and $T_C$ intersect at $q$ at an angle $\theta_q, $ and measured anticlockwise from $T_A$ to $T_C$. Let $\ell_{pq}$ be the signed distance from $p$ to $q$ along $T_A$. Throughout this computation, we assume that $q$ lies in the positive direction from $p$ along $T_A$ (see Remark \ref{rem:intersection position} for the explanation). See the left-most picture of Figure \ref{fig: 3 int geod x_z}.
\begin{lemma}\label{lem:triplec1}
With the above description of $A,B,C\in SL_2(\mathbb{R})$,
$$\begin{aligned}
   \tr(N_XN_Y)=& \ 2\epsilon_p(T_A,T_B)\cos\theta_p,\\
    \tr(N_XN_Z)=&\ 2\epsilon_q(T_A,T_C)\cos\theta_q,\\
    \tr(N_YN_Z)=& \ 2\epsilon_p(T_A,T_B)\epsilon_q(T_A,T_C)(\cos\theta_p\cos\theta_q
+
\cosh\ell_{pq}\,\sin\theta_p\sin\theta_q),\\
\tr(N_XN_YN_Z)=&-2\epsilon_p(T_A,T_B)\epsilon_q(T_A,T_C)\sinh\ell_{pq}\,\sin\theta_p\sin\theta_q
\end{aligned} 
$$
\begin{proof}
     We use the same notation as in Lemma \ref{lem:double} and Lemma \ref{lem: rep hyp elmt}. 
    Using Lemma \ref{lem:double}, we get the first identity. The hyperbolic translation along $Y$-axis by signed
distance $\ell_{pq}$ is
$$
D_{\ell_{pq}}=
\begin{pmatrix}
e^{\ell_{pq}/2}&0\\
0&e^{-\ell_{pq}/2}
\end{pmatrix}.
$$
Therefore  $e^{\ell_{pq}} i$ is a lift of $q$ which 
is at signed distance $\ell_{pq}$ from $i$ along the axis of $A$.
\begin{figure}[h]
  \centering
    \includegraphics[ width=2.8 cm, angle=90]{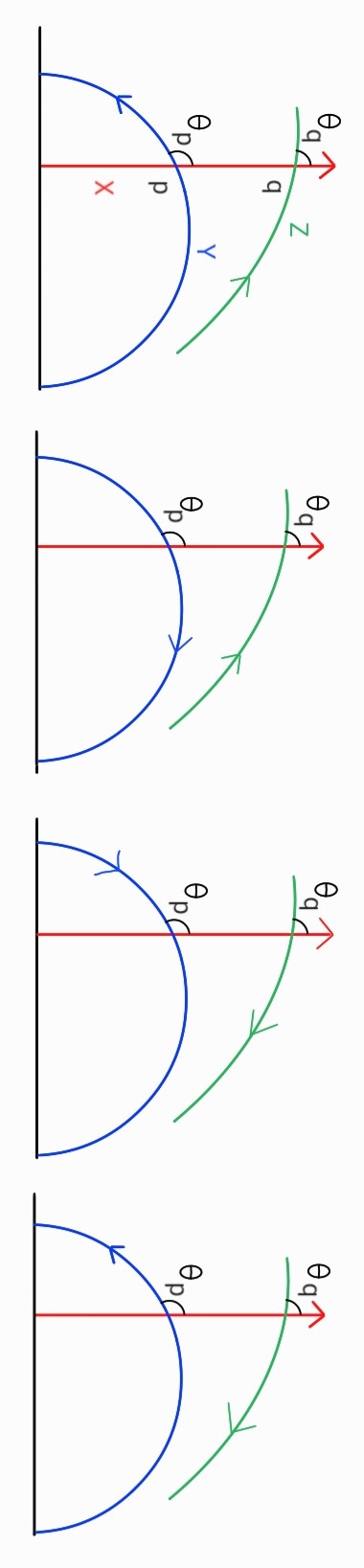}
      \caption{$T_A$ intersects $T_B$; $T_A$ intersects $T_C$}
     \label{fig: 3 int geod x_z}
\end{figure}
The axis of  $C$  passes through
$D_{\ell_{pq}}(i).$

Now we consider the case $\epsilon_p(T_A,T_B)=1$ and $\epsilon_p(T_A,T_C)=1$. The situation is depicted at the first figure of Figure \ref{fig: 3 int geod x_z}. So $C$ is obtained by conjugating the isometry with translation length $c$ whose axis  $i$ making angle $\theta_q$ with $Y$-axis by $D_{\ell_{pq}}$. 
Therefore, we get
$
N_Z=D_{\ell_{pq}} N_{\theta_q} D_{\ell_{pq}}^{-1},
$
where $$
N_{\theta_q}=
\begin{pmatrix}
\cos\theta_q&-\sin\theta_q\\
-\sin\theta_q&-\cos\theta_q
\end{pmatrix}.
$$
Thus
$$
\begin{aligned}
N_Z
&=
\begin{pmatrix}
e^{\ell_{pq}/2}&0\\
0&e^{-\ell_{pq}/2}
\end{pmatrix}
\begin{pmatrix}
\cos{\theta_q}&-\sin{\theta_q}\\
-\sin{\theta_q}&-\cos{\theta_q}
\end{pmatrix}
\begin{pmatrix}
e^{-\ell_{pq}/2}&0\\
0&e^{\ell_{pq}/2}
\end{pmatrix} \\
&=
\begin{pmatrix}
\cos{\theta_q}&-e^{\ell_{pq}}\sin{\theta_q}\\
-e^{-\ell_{pq}}\sin{\theta_q}&-\cos{\theta_q}
\end{pmatrix}.
\end{aligned}
$$
Then a direct multiplication gives
\begin{align*}
   & {\tr}(N_XN_Z)=2\cos\theta_q,\
\operatorname{tr}(N_YN_Z)
=
2(\cos\theta_p\cos{\theta_q}
+
\cosh\ell_{pq}\,\sin\theta_p\sin{\theta_q}),\\
&\operatorname{tr}(N_XN_YN_Z)
=-
2\sinh\ell_{pq}\,\sin\theta_p\sin{\theta_q}.
\end{align*}
All other cases are shown in Figure \ref{fig: 3 int geod x_z} and are obtained by reversing the directions of $T_A$, $T_B$ and $T_C$. 
For example, if the direction of $T_A$ and $T_B$ are same as in the first picture of Figure \ref{fig: 3 int geod x_z} and $T_C$ is in opposite direction
then, we obtain the fourth picture and we replace $N_Z$ by $-N_Z$ and a similar computation shows
\begin{align*}
    {\tr}(N_XN_Z)=-2\cos\theta_q,\
&\operatorname{tr}(N_YN_Z)
=-
2(\cos\theta_p\cos{\theta_q}
+
\cosh\ell_{pq}\,\sin\theta_p\sin{\theta_q}),\\ \text{and }
\operatorname{tr}(N_XN_YN_Z)
&=
2\sinh\ell_{pq}\,\sin\theta_p\sin{\theta_q}.
\end{align*}
Considering all possible cases, we get the identities.
\end{proof}
\begin{corollary}\label{cor: sum ABC} Using the notation in Lemma \ref{lem:triplec1}, we have
  \begin{align*}
   \frac12{\tr}(ABC)
 =&
\ \varepsilon_A\varepsilon_B\varepsilon_C (C_aC_bC_c
+
\epsilon_p(T_A,T_B) S_aS_bC_c\cos\theta_p
+
\epsilon_q(T_A,T_C) S_aC_bS_c\cos\theta_q \\
&+
 C_aS_bS_c\epsilon_p(T_A,T_B)\epsilon_q(T_A,T_C)
\left(
\cos\theta_p\cos\theta_q
+
\cosh\ell_{pq}\,\sin\theta_p\sin\theta_q
\right) \\
&-
\epsilon_p(T_A,T_B)\epsilon_q(T_A,T_C)S_aS_bS_c
\sinh\ell_{pq}\,\sin\theta_p\sin\theta_q). 
\end{align*}   
  \end{corollary}
\end{lemma}
\begin{proof}

By Lemma \ref{lem:gen tr 
ident},
\begin{align*}
    \frac12{\tr}(ABC)
 =&
\ \varepsilon_A\varepsilon_B\varepsilon_C (C_aC_bC_c
+
 S_aS_bC_c\frac12\tr(N_XN_Y)
+
S_aC_bS_c\frac12\tr(N_XN_Z) \\
&+
 C_aS_bS_c
\frac12\tr(N_YN_Z) 
+
S_aS_bS_c
\frac12\tr(N_XN_YN_Z)).
\end{align*}
We substitute all values using Lemma \ref{lem:triplec1} to get the desired identity.
\end{proof}
\begin{lemma}\label{lem: sum 1}
 With the same assumption as above, we have $$
\begin{aligned}
&\epsilon_p(T_A,T_B)\epsilon_q(T_A,T_C)\left({\tr}(
CAB-CB^{-1}A^{-1}-CAB^{-1}+CBA^{-1}
)\right) \\
&=
8
\varepsilon_A\varepsilon_B\varepsilon_C \Big(
 C_aS_bS_c
 \cos\theta_p\cos\theta_q
+S_bS_c\cosh\left(\ell_{pq}-\frac{a}{2}\right)
\sin\theta_p\sin\theta_q
 \Big).
\end{aligned}
$$
\end{lemma}
\begin{proof}
    Recall that trace is invariant under cyclic permutation and is invariant under inversion in $SL_2(\mathbb R)$. Therefore the above expression is same as $$\epsilon_p(T_A,T_B)\epsilon_q(T_A,T_C)\left({\tr}(
ABC-ABC^{-1}-AB^{-1}C+AB^{-1}C^{-1}
)\right).$$ 
The identities for $\frac{1}{2}\tr(ABC^{-1})$, $\frac{1}{2}\tr(AB^{-1}C)$ and $\frac{1}{2}\tr(AB^{-1}C^{-1})$ can be obtained from Corollary  \ref{lem: sum 1} in two equivalent ways. On the one hand, they follow by replacing $N_Z$ by $-N_Z$ and $N_Y$ by $-N_Y$, and both $N_Y$ and $N_Z$ with their negatives, respectively. On the other hand, they are obtained by replacing $C$ by $C^{-1}$, $B$ by $B^{-1}$, and both $B$ and $C$ with their inverses, respectively. In the latter approach, we use the relations $\epsilon_p(T_A,T_{B^{-1}})=-\epsilon_p(T_A,T_B)$ and $\epsilon_p(T_A,T_{C^{-1}})=-\epsilon_p(T_A,T_C)$. Therefore,
\begin{align*}
   \frac12{\tr}(ABC^{-1})
 &=
\varepsilon_A\varepsilon_B\varepsilon_C( C_aC_bC_c
+
\epsilon_p(T_A,T_B) S_aS_bC_c\cos\theta_p
-
\epsilon_q(T_A,T_C) S_aC_bS_c\cos\theta_q \\
&-
 C_aS_bS_c\epsilon_p(T_A,T_B)\epsilon_q(T_A,T_C)
\left(
\cos\theta_p\cos\theta_q
+
\cosh\ell_{pq}\,\sin\theta_p\sin\theta_q
\right) \\
&  +
\epsilon_p(T_A,T_B)\epsilon_q(T_A,T_C)S_aS_bS_c
\sinh\ell_{pq}\,\sin\theta_p\sin\theta_q), \\
   \frac12{\tr}(AB^{-1}C)
 &=
\varepsilon_A\varepsilon_B\varepsilon_C (C_aC_bC_c
-
\epsilon_p(T_A,T_B) S_aS_bC_c\cos\theta_p
+
\epsilon_q(T_A,T_C) S_aC_bS_c\cos\theta_q \\
& \quad \quad-
C_aS_bS_c\epsilon_p(T_A,T_B)\epsilon_q(T_A,T_C)
\left(
\cos\theta_p\cos\theta_q
+
\cosh\ell_{pq}\,\sin\theta_p\sin\theta_q
\right) \\
& \quad \quad +
\epsilon_p(T_A,T_B)\epsilon_q(T_A,T_C)S_aS_bS_c
\sinh\ell_{pq}\,\sin\theta_p\sin\theta_q), \\
   \frac12{\tr}(AB^{-1}C^{-1})
 & =
\varepsilon_A\varepsilon_B\varepsilon_C (C_aC_bC_c
-
\epsilon_p(T_A,T_B) S_aS_bC_c\cos\theta_p
-
\epsilon_q(T_A,T_C) S_aC_bS_c\cos\theta_q \\
&  \quad \quad+
C_aS_bS_c\epsilon_p(T_A,T_B)\epsilon_q(T_A,T_C)
\left(
\cos\theta_p\cos\theta_q
+
\cosh\ell_{pq}\,\sin\theta_p\sin\theta_q
\right) \\
& \quad \quad -
\epsilon_p(T_A,T_B)\epsilon_q(T_A,T_C)S_aS_bS_c
\sinh\ell_{pq}\,\sin\theta_p\sin\theta_q).
\end{align*} 
Consequently
\begin{align*}
&\epsilon_p(T_A,T_B)\epsilon_q(T_A,T_C)({\tr}(
ABC-ABC^{-1}-AB^{-1}C+AB^{-1}C^{-1}
))\\
&=
8S_bS_c
\varepsilon_A\varepsilon_B\varepsilon_C\Big(
C_a
\left(
\cos\theta_p\cos\theta_q
+
\cosh\ell_{pq}\,\sin\theta_p\sin\theta_q
\right)
-
S_a\sinh\ell_{pq}\,\sin\theta_p\sin\theta_q
\Big)\\
&=8
 \varepsilon_A\varepsilon_B\varepsilon_C \Big(
 C_aS_bS_c
 \cos\theta_p\cos\theta_q+S_bS_c\left(C_a\cosh\ell_{pq}-S_a\sinh\ell_{pq}\right)\sin\theta_p\cos\theta_q\Big)
\\&=
8
 \varepsilon_A\varepsilon_B\varepsilon_C\Big(
 C_aS_bS_c
 \cos\theta_p\cos\theta_q
+S_bS_c\cosh\left(\ell_{pq}-\frac{a}{2}\right)
\sin\theta_p\sin\theta_q
 \Big). 
\end{align*}
This completes the proof.
\end{proof}
\begin{remark}\label{rem: sum 1}
    Recall that $\epsilon_q(T_A,T_C)=-\epsilon_q(T_C,T_A)$. Therefore, the identity of Corollary \ref{lem: sum 1} can be written  as 
\begin{align*}
&\epsilon_p(T_A,T_B)\epsilon_q(T_C,T_A)({\tr}(
ABC-ABC^{-1}-AB^{-1}C+AB^{-1}C^{-1}
))\\
&= -8\varepsilon_A\varepsilon_B\varepsilon_C\Big(
 C_aS_bS_c
 \cos\theta_p\cos\theta_q
+S_bS_c\cosh\left(\ell_{pq}-\frac{a}{2}\right)
\sin\theta_p\sin\theta_q
 \Big).
\end{align*}
\end{remark}
\subsubsection{Case II: $T_B$ intersects $T_C$ } 
Suppose the axes $T_A$ and $T_B$ intersect at $r$ at an angle $\theta_r$, measured anticlockwise from $A$ to $B$; the axes $T_B$ and $T_{C}$ intersect at $s$ at an angle $\theta_s, $ and measured anticlockwise from $B$ to $C$. Let $\ell_{sr}$ be the signed distance from $s$ to $r$ along $T_B$. Throughout this computation, we assume that $s$ lies in the negative direction from $r$ along $T_B$ (see Remark \ref{rem:intersection position} for the explanation). See the first figure of Figure \ref{fig: 3 int geod y_z}.
\begin{figure}[ht]
      \centering
     \includegraphics[width=8 cm]{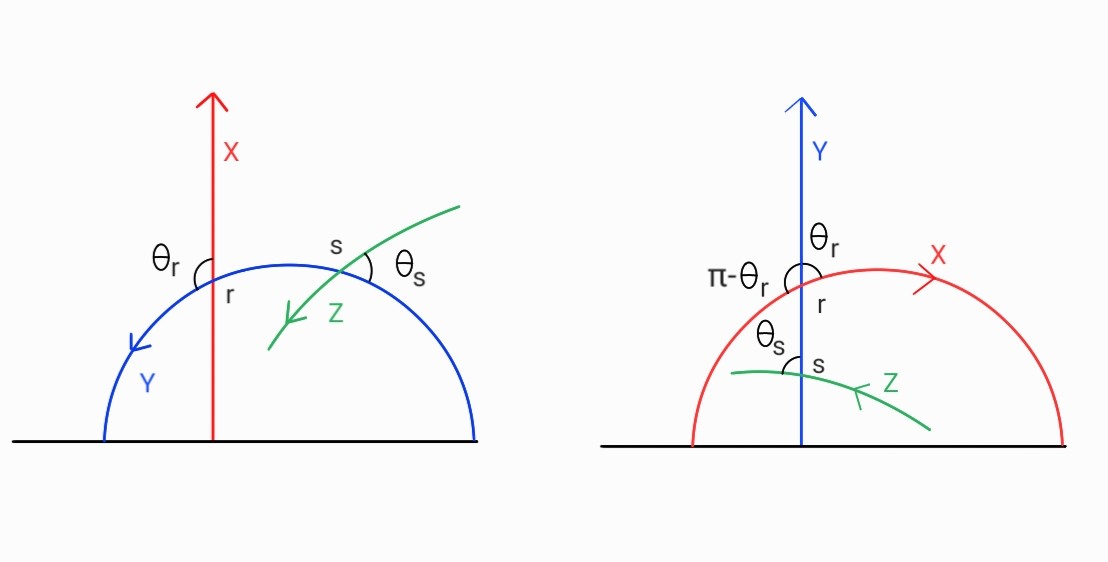}
      \caption{$T_A$ intersects $T_B$; $T_B$ intersects $T_C$}
      \label{fig: 3 int geod y_z}
 \end{figure}
\begin{lemma}\label{lem:triplec2}
    With the above description of $A,B,C\in SL_2(\mathbb{R})$,
     \begin{align*}
   \tr(N_{X}N_{Y})=& 
   2\epsilon_r(T_A, T_B)\cos\theta_r,
   \\
    \tr(N_{Y}N_{Z})=&  2\epsilon_s(T_{B},T_C)\cos\theta_s,\\
    \tr(N_{X}N_{Z})=& -2\epsilon_r(T_A,T_B)\epsilon_s(T_{C},T_B)(\cos\theta_r\cos\theta_s
-
\cosh\ell_{sr}\,\sin\theta_r\sin\theta_s)\\
\tr(N_{X}N_{Y}N_{Z})=& \ 2\epsilon_r(T_A,T_B)\epsilon_s(T_{C},T_B)\sinh(-\ell_{sr})\sin\theta_r\sin\theta_s
\end{align*}
\end{lemma}
\begin{proof}
We conjugate simultaneously $X$, $Y$, $Z$, thereby transforming the first diagram in Figure \ref{fig: 3 int geod y_z} into the second diagram. 
Comparing the resulting diagram with the first diagram in Figure \ref{fig: 3 int geod x_z}, we observe that, in the present case, $X$, $Y$, $Z$ are equivalent respectively, to $Y^{-1}$, $X$, and $Z$ of Case I.
Also  $\theta_p $ is equivalent to $\pi-\theta_r$ and $\ell_{pq}$ is equivalent to $-\ell_{sr}$. Therefore
$$ 
N_X=
\begin{pmatrix}
\cos\theta_r &\sin\theta_r \\
\sin\theta_r &-\cos\theta_r
\end{pmatrix},
N_Y=
\begin{pmatrix}
1&0\\
0&-1
\end{pmatrix}, N_Z=
\begin{pmatrix}
\cos{\theta_s}&-e^{(-\ell_{sr})}\sin{\theta_s}\\
-e^{-(-\ell_{sr})}\sin{\theta_s}&-\cos{\theta_s}
\end{pmatrix}.
$$
A direct computation gives
 \begin{align*}
   \tr(N_{X}N_{Y})=&\
   2\epsilon_r(T_A, T_B)\cos\theta_r,\,\,
    \tr(N_{Y}N_{Z})= \ 2\epsilon_s(T_{B},T_C)\cos\theta_s,\\
    \tr(N_{X}N_{Z})=& \ 2\epsilon_r(T_A,T_B)\epsilon_s(T_{B},T_C)(\cos\theta_r\cos\theta_s
-
\cosh\ell_{sr}\,\sin\theta_r\sin\theta_s)\\
=&- 2\epsilon_r(T_A,T_B)\epsilon_s(T_{C},T_B)(\cos\theta_r\cos\theta_s-
\cosh\ell_{sr}\,\sin\theta_r\sin\theta_s)\\
\text{ and }
\tr(N_{X}N_{Y}N_{Z})=& \ -2\epsilon_r(T_A,T_B)\epsilon_s(T_{B},T_C)\sinh(-\ell_{sr})\,\sin\theta_r\sin\theta_s\\
=&\ 2\epsilon_r(T_A,T_B)\epsilon_s(T_{C},T_B)\sinh(-\ell_{sr})\,\sin\theta_r\sin\theta_s .
\end{align*} 
This completes the proof.
\end{proof}
\begin{corollary}\label{cor: sum ABC_2}
   Using the notation in Lemma \ref{lem:triplec2} 
\begin{align*}
    \frac12{\tr}(ABC)
    =&\ \varepsilon_{A}\varepsilon_{B}\varepsilon_{C}(C_{a}C_{b}C_{c}
+ S_{a}S_{b}C_{c}\epsilon_r(T_{A},T_{B})\cos\theta_r+ C_{a}S_{b}S_{c}\epsilon_s(T_{B},T_{C})\cos\theta_s
\\
&- S_{a}C_{b}S_{c}\epsilon_r(T_{A},T_{B})\epsilon_s(T_{C},T_{B})(\cos\theta_r\cos\theta_s
-
\cosh\ell_{sr} \,\sin\theta_r\sin\theta_s)\\
&+S_{a} S_{b} S_{c}\epsilon_r(T_{A},T_{B})\epsilon_s(T_{C},T_{B})\sinh(-\ell_{sr})\,\sin\theta_r\sin\theta_s).
\end{align*}
\end{corollary}
\begin{lemma}\label{lem: sum 2}
   With the same assumption, we have
    \begin{align*}
&\epsilon_r(T_A,T_B)\epsilon_s(T_{C},T_B)({\tr}(
ABC-ABC^{-1}+AB^{-1}C-AB^{-1}C^{-1}
)) \\
&=8\varepsilon_{A}\varepsilon_{B}\varepsilon_{C}
 \Big(- S_{a}C_{b}S_{c}\cos\theta_r\cos\theta_s+S_{a}S_{c}\cosh\big(\ell_{rs}-\frac{b}{2}\big)\sin\theta_r\sin\theta_s\Big).
 \end{align*}

\end{lemma}
\begin{proof}
Now  by Lemma \ref{lem:gen tr 
ident}, Lemma \ref{lem: sum 2} and a similar computation as in Lemma \ref{lem: sum 1} gives
\begin{align*}
    \frac 12 \tr(ABC)&-\frac 12 \tr(ABC^{-1})+ \frac 12 (AB^{-1}C)- \frac 12 (AB^{-1}C^{-1})\\
    = &4\varepsilon_{A}\varepsilon_{B}\varepsilon_{C}\Big(- S_{a}C_{b}S_{c}\epsilon_r(T_{A},T_{B})\epsilon_s(T_{C},T_{B})(\cos\theta_r\cos\theta_s
-
\cosh\ell_{sr}\,\sin\theta_r\sin\theta_s)\\
 &+ S_{a} S_{b} S_{c}\epsilon_r(T_{A},T_{B})\epsilon_s(T_{C},T_{B})\sinh(-\ell_{sr})\,\sin\theta_r\sin\theta_s\Big).
 \end{align*}
Further simplification together with the fact $\ell_{rs}+\ell_{sr}=b$ yields the identity.
\end{proof}

\begin{lemma}\label{Lem: Sign}
For intersecting godesics  $x$ and $y$ at $p$ we have Then $\varepsilon_{x*_py}=\varepsilon_x\varepsilon_y$ and $\varepsilon_{x*_py^{-1}}=\varepsilon_x\varepsilon_y$.
\end{lemma}
\begin{proof}
 This follows from Remark \ref{rem:sign}. 
 \end{proof}
\section{Poisson bracket of trace and length functions}\label{sec:Poisson bracket of trace and length functions}
In this section, we connect Wolpert's description of the Weil-Petersson symplectic form on the Teichm\"{u}ller space and the Poisson bracket of length functions, using Goldman's topological formalism. In \cite{MR762512}, Goldman proved that  Atiyah-Bott-Goldman symplectic form is a scalar multiple of the Weil-Petersson symplectic form.  As the exact value of the constant is not relevant to our current discussion, we prove all the results in terms of Atiyah-Bott-Goldman symplectic form. For a detailed discussion about the exact value of the constant, see \cite{MR4911761}.

Throughout this section, we fix a point $\Phi\in \T$ and describe all the geometric objects with respect to $\Phi$. As described in Section \ref{sec:basic}, we consider $\Phi$ as an element of $\M_W$ and choose a lift of $\Phi$ to $\M$ which varies continuously on $\T$. We also fix this lift for our discussion in this section. 

Suppose $x,y$ and $z$ are three oriented geodesics on  $\Sigma$ with respect to $\Phi$ such that:  $x$ intersects $y$ at a point $p$ with intersection angle $\theta_p$,
   
Identify the $\pi_1(\Sigma,p)$ as a subgroup of $PSL(2,\mathbb R)$. Suppose $\tilde A,\tilde B, $ and $\tilde C\in PSL(2,\mathbb R)$ represent the deck transformations associated to the oriented geodesics $x,y, z$, with translation lengths $a=\ell_x, b=\ell_y $ and $  c=\ell_z$ respectively. Let $A,B$ and $C\in SL_2(\R)$ be their lifts to $SL_2(\R)$  with respect to the lift of $\Phi$.  By Lemma \ref{lem: rep hyp elmt} we  write $
A=\varepsilon_x(C_aI+S_aN_X), B=\varepsilon_y(C_bI+S_bN_Y),$ and $ C=\varepsilon_z(C_cI+S_cN_Z).
$ where $\varepsilon_x=\varepsilon_A, \varepsilon_y=\varepsilon_B$ and $\varepsilon_z=\varepsilon_C$.

For any intersection points $p \in x\cap y$ and $q\in x\cap z$  we choose normalized lifts of $T_A, T_B$ and $T_C$ such that $T_A$ is the $Y$ axis and $p=i$. Similarly, for any intersection points $r \in x\cap y$ and $s\in y\cap z$, we choose normalized lifts of $T_A, T_B$ and $T_C$ such that $T_A$ is the $Y$ axis and $p=i$. Abusing notation, we denote the lift of a point by the same notation.  

\begin{remark}\label{rem:intersection position}
With the above notation, it follows from Figure \ref{fig:int z to xy} that $q$ lies in the positive direction from $p$ along $T_A$ and $s$ lies in the negative direction from $r$ along $T_B$. 
\end{remark}

\subsection{Poisson bracket of trace functions} We compute the Poisson bracket of trace functions using Goldman's theorem \ref{thm:Goldman}. 

\begin{theorem}\label{thm:cosine trace}
    For two intersecting geodesics $x$ and $y$ on $\Sigma$, we have $$\{\tr_x,\tr_y\}=  2\sinh\frac{a}{2}\sinh \frac{b}{2}\sum_{p\in x\cap y}  \cos\theta_p.$$
\end{theorem}
\begin{proof}
\begin{align*}
  \{\tr_x,\tr_y\}=&\frac 12\sum_{p\in x\cap y}\epsilon_p(x,y) (\tr_{x*_py}-\tr_{x*_py^{-1}}) \text{ (by Theorem \ref{thm:Goldman})}\\
  =&\frac 12\sum_{p\in x\cap y}\epsilon_p(x,y) \big(\varepsilon_{x*_py}\tr(AB)-\varepsilon_{x*_py^{-1}}\tr(AB^{-1})\big)\\
  =& \frac 12 \varepsilon_x\varepsilon_y \sum_{p\in x\cap y}\epsilon_p(x,y) \big(\tr(AB)-\tr(AB^{-1})\big)  \text{ (by Lemma \ref{Lem: Sign}) \quad}\\
=& \frac 12 \varepsilon_x\varepsilon_y \sum_{p\in x\cap y} 4\varepsilon_A\varepsilon_B \sinh\frac{a}{2}\sinh \frac{b}{2} \cos\theta_p \text{ (by Corollary \ref{cor:double})}.
\end{align*} 
This completes the proof.
\end{proof}


\begin{figure}[ht]
    \includegraphics[height= 4cm, width= 8cm]{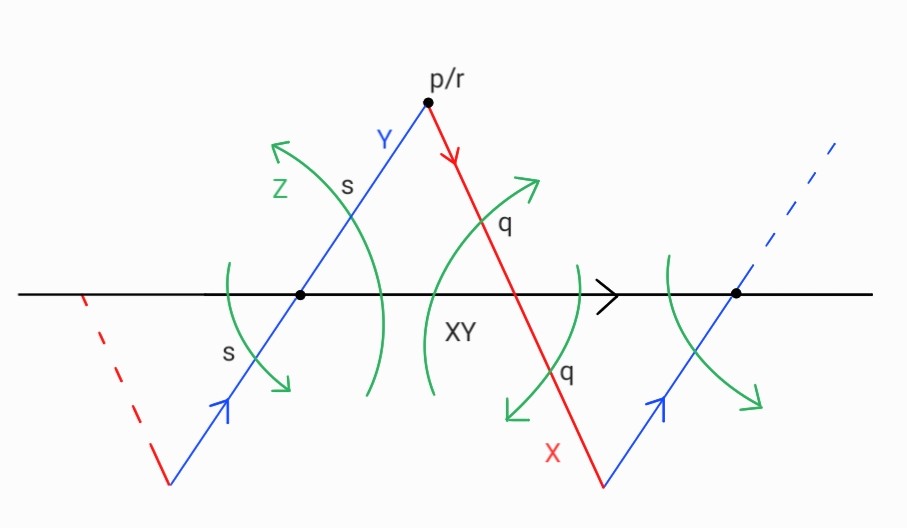}
    \caption{}
    \label{fig:int z to xy}
\end{figure}

\begin{lemma}\label{lem: z with x*y int}
    With the above description of $x,y,z$, if $z$ intersects $x*_py$, $z$ must intersect either $x$ or $y$. Similarly, if $z$ intersects $x*_py^{-1}$, $z$ must intersects either $x$ or $y$.
\end{lemma}

\begin{proof}
    The convexity of the region enclosed by $X,Y$ and $XY$ in Figure \ref{fig:sign x y z}, the conclusion follows.
\end{proof}
\begin{figure}[h]
    \includegraphics[height= 4cm, width= 6.2cm]{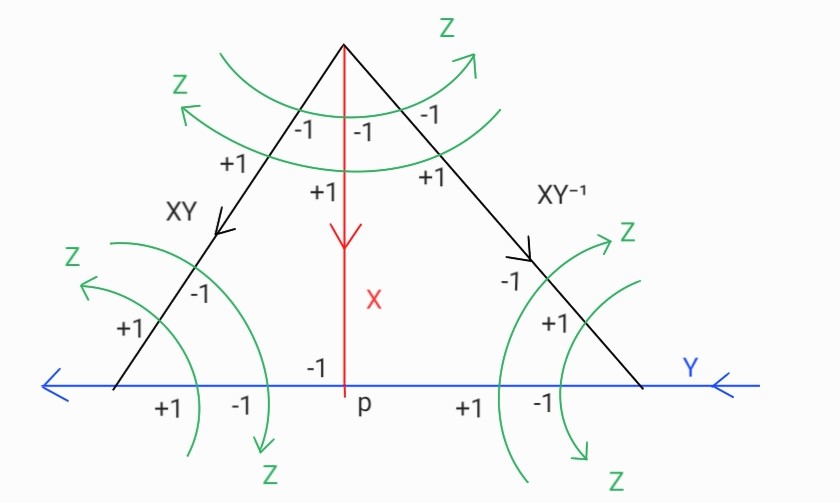}
    \includegraphics[height= 4cm, width= 6cm]{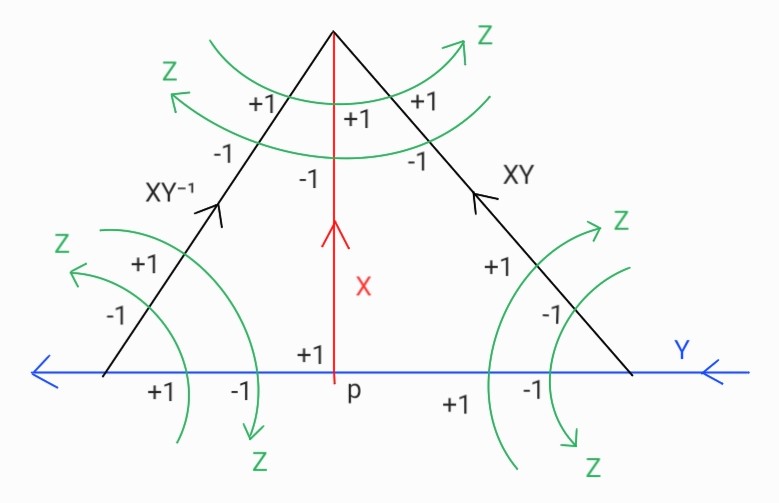}
    \caption{}
    \label{fig:sign x y z}
\end{figure}

\begin{lemma}\label{lem: sign pat}
    With the above description of $x,y,z$,
    \[
\epsilon_q(z,x*_py)=
\begin{cases}
\epsilon_q(z,x), & q\in z\cap x, \\
\epsilon_q(z,y),  & q\in z\cap y.
\end{cases}
\quad
\text{ and}
\quad
\epsilon_q(z,x*_py^{-1})=
\begin{cases}
\epsilon_q(z,x), & q\in z\cap x, \\
-\epsilon_q(z,y),  & q\in z\cap y.
\end{cases}
\]
    Therefore, both are independent of the sign $\epsilon_p.$
\end{lemma}
\begin{proof}

   All possible intersection patterns between $x,y$ and $z$ are shown in Figure \ref{fig:sign x y z}. The proof follows from Figure \ref{fig:sign x y z}.
\end{proof}
\begin{theorem}\label{thm:sine trace}
    Suppose $x,y$ and $z$ are three closed geodesics on $\Sigma$ with $\ell_x=a$, $\ell_y=b$ and $\ell_z=c$. Denote $C_u=\cosh\frac u2$ and $S_u=\sinh\frac u2.$ We have
    \begin{align*}
        &\{\tr_z, \{\tr_x, \tr_y\}\}\\ = & \ -2C_aS_b S_c\sum_{p\in x\cap y,q\in z\cap x}\cos\theta_p\cos\theta_q-  2 S_aC_bS_c\sum_{r\in x\cap y,s\in z\cap y}\cos\theta_r\cos\theta_s\\
 &- 2S_aS_bS_c\sum_{p\in x\cap y,q\in z\cap x}\frac{\cosh\big(\ell_{pq}-\frac{a}{2}\big)}{S_a}\sin\theta_p\sin\theta_q
 + 2S_aS_bS_c\sum_{r\in x\cap y,s\in z\cap y}\frac{\cosh\big(\ell_{rs}-\frac{b}{2}\big)}{S_b}\sin\theta_r\sin\theta_s.
    \end{align*}
  where $\ell_{pq}$ (respectively $\ell_{rs}$) denotes the distance from $p$ to $q$ (respectively $r$ to $s$) along $x$ (respectively along $y$). 
\end{theorem}
 \begin{proof} 
 Using Theorem \ref{thm:Goldman} and Leibniz rule, we have 
\begin{flalign*}
  ~ \{\tr_z, \{\tr_x, \tr_y\}\} 
 &=   \{\tr_z,\frac 12\sum_{p\in x\cap y}\epsilon_p(x,y) (\tr_{x*_py}-\tr_{x*_py^{-1}})\} \\
&  = \frac 12\sum_{p\in x\cap y}\epsilon_p(x,y)\left[\{\tr_z, \tr_{x*_py}\}-\{\tr_z,\tr_{x*_py^{-1}}\}\right]
\end{flalign*}

\begin{align*}
& =   \frac 12 \sum_{p\in x\cap y} \epsilon_p(x,y)\bigg[\frac 12\sum_{q\in z\cap (x*_py)}\epsilon_q(z, x*_py)(\tr_{z*_qx*_py}-\tr_{z*_q(x*_py)^{-1}})\bigg] \\  
   &   \hspace{1.5 cm}- \ \frac 12 \sum_{p\in x\cap y} \epsilon_p(x,y)\bigg[ \frac 12 \sum_{q\in z\cap (x*_py^{-1})}\epsilon_q(z, x*_py^{-1})(\tr_{z*_qx*_py^{-1}}-\tr_{z*_q(x*_py^{-1})^{-1}})\bigg] .
   \end{align*}

By Lemma \ref{lem: sign pat} and Lemma \ref{lem: z with x*y int} the right-hand side is equal to

   \begin{align*}
   & =  \frac 12 \sum_{p\in x\cap y} \epsilon_p(x,y)\bigg[ \frac 12\sum_{q\in z\cap x}\epsilon_q(z, x)(\tr_{z*_qx*_py}-\tr_{z*_q(x*_py)^{-1}})\\  
& \hspace{3.5 cm} + \frac 12\sum_{q\in z\cap y}\epsilon_q(z, y)(\tr_{z*_qx*_py}-\tr_{z*_q(x*_py)^{-1}})\bigg]\\
       & \quad   -\frac 12 \sum_{p\in x\cap y} \epsilon_p(x,y)\bigg[\ \frac 12 \sum_{q\in z\cap x}\epsilon_q(z, x)(\tr_{z*_qx*_py^{-1}}-\tr_{z*_q(x*_py^{-1})^{-1}}) \\
       & \hspace{3.5 cm} + \frac 12
        \sum_{q\in z\cap y}-\epsilon_q(z, y)(\tr_{z*_qx*_py^{-1}}-\tr_{z*_q(x*_py^{-1})^{-1}})\bigg]       
  \\&= \frac 14 \sum_{p\in x\cap y} \epsilon_p(x,y) \sum_{q\in z\cap x}\epsilon_q(z, x) \left [ \tr_{z*_qx*_py}-\tr_{z*_qy^{-1}*_px^{-1}} - \tr_{z*_qx*_py^{-1}} + \tr_{z*_q y*_px^{-1}} \right]\\
         &\quad +  \frac 14 \sum_{p\in x\cap y} \epsilon_p(x,y) \sum_{q\in z\cap y}\epsilon_q(z, y) \left [ \tr_{z*_qx*_py}-\tr_{z*_qy^{-1}*_px^{-1}} + \tr_{z*_qx*_py^{-1}} - \tr_{z*_q y*_px^{-1}} \right]  .   
        \end{align*}
       By Lemma \ref{Lem: Sign} and Remark \ref{rem:sign}, 
    $\varepsilon_{zxy}, \varepsilon_{zy^{-1}x^{-1}},\varepsilon_{zxy^{-1}},\varepsilon_{zyx^{-1}}$, are all equal to $\varepsilon_x\varepsilon_y\varepsilon_z$. Therefore, the right-hand side is 
    
        \begin{align*}
       = \ & \frac 14\varepsilon_x\varepsilon_y\varepsilon_z \sum_{p\in x\cap y, q\in z\cap x} \epsilon_p(x,y)\epsilon_q(z, x) \left [ \tr(CAB)-\tr(CB^{-1}A^{-1}) - \tr(CAB^{-1}) + \tr(CBA^{-1}) \right]\\
         &+\frac 14 \varepsilon_x\varepsilon_y\varepsilon_z \sum_{r\in x\cap y,s\in z\cap y} \epsilon_r(x,y)\epsilon_s(z, y) \left [ \tr(CAB)-\tr(CB^{-1}A^{-1}) + \tr(CAB^{-1}) - \tr(CBA^{-1})\right].
         \end{align*} 
          Substituting the expressions from Remark \ref{rem: sum 1} and Lemma \ref{lem: sum 2} we obtain
         \begin{align*}
    =&\frac 14 \varepsilon_x\varepsilon_y\varepsilon_z\sum_{p\in x\cap y,q\in z\cap x}  -8\varepsilon_A\varepsilon_B\varepsilon_C\Big(
 C_aS_bS_c
 \cos\theta_p\cos\theta_q
+S_bS_c\cosh\left(\ell_{pq}-\frac{a}{2}\right)
\sin\theta_p\sin\theta_q
 \Big)\\
 &+ \frac 14 \varepsilon_x\varepsilon_y\varepsilon_z\sum_{r\in x\cap y,s\in z\cap y}8\varepsilon_A\varepsilon_B\varepsilon_C\Big(- S_{a}C_{b}S_{c}\cos\theta_r\cos\theta_s+S_{a}S_{c}\cosh\big((-\ell_{sr})+\frac{b}{2}\big)\sin\theta_r\sin\theta_s\Big).
 \end{align*} 
 Apply $\varepsilon_x^2=\varepsilon_y^2=\varepsilon_z^2=1$ and a further simplification yields
 \begin{align*}
 & \{\tr_z, \{\tr_x, \tr_y\}\}= -2C_aS_b S_c\sum_{p\in x\cap y,q\in z\cap x}\cos\theta_p\cos\theta_q-  2 S_aC_bS_c\sum_{r\in x\cap y,s\in z\cap y}\cos\theta_r\cos\theta_s\\
 &- 2S_aS_bS_c\sum_{p\in x\cap y,q\in z\cap x}\frac{\cosh\big(\ell_{pq}-\frac{a}{2}\big)}{S_a}\sin\theta_p\sin\theta_q
 + 2S_aS_bS_c\sum_{r\in x\cap y,s\in z\cap y}\frac{\cosh\big(\ell_{rs}-\frac{b}{2}\big)}{S_b}\sin\theta_r\sin\theta_s.
 \end{align*} 
 This completes the proof.
\end{proof}

\subsection{Poisson bracket of length functions} We compute the Poisson bracket of length functions using Goldman's theorem. 
Observe that the length function does not depend on the sign of the trace of the lift. Hence, for the discussion below, we assume all the matrices to have positive trace.
\begin{theorem}[Wolpert's cosine formula] \label{thm:cosine form} Suppose $x$ and $y$ are two closed geodesics on $\Sigma$ with $\ell_x=a$, $\ell_y=b$. We have
$$\{\ell_x,\ell_y\}=2\sum\limits_{p\in x\cap y}\cos{\theta}_p.$$
\end{theorem}
\begin{proof}
\begin{align*}
    \{\tr_x,\tr_y\}=& \   \{2\varepsilon_x\cosh\frac{\ell_x}{2}, 2\varepsilon_y\cosh\frac{l_y}{2}\}\\
    =&\ 4\varepsilon_x\varepsilon_y \{\cosh\frac{\ell_x}{2},\cosh\frac{\ell_y}{2}\}\\
    =&\  4\varepsilon_x\varepsilon_y  \sinh{\frac{\ell_x}{2}}\sinh{\frac{\ell_y}{2}}\{\frac{\ell_x}{2}, \frac{\ell_y}{2}\}\\
    =&\  \varepsilon_x\varepsilon_y  \sinh{\frac{\ell_x}{2}}\sinh{\frac{\ell_y}{2}}\{\ell_x,\ell_y\}\\
    =&\  2\sinh\frac{a}{2}\sinh \frac{b}{2}\sum_{p\in x\cap y}  \cos\theta_p  \text{  (by Theorem \ref{thm:cosine trace}) }.
\end{align*}  
Combining the last two lines and substituting $\varepsilon_x=\varepsilon_A=1$ and $\varepsilon_y=\varepsilon_B=1$ we obtain the formula. 
\end{proof}

\begin{theorem}[Wolpert's sine formula]\label{thm:sine form}
   Suppose $x,y$ and $z$ are three closed geodesics on $\Sigma$ with $\ell_x=a$, $\ell_y=b$ and $\ell_z=c$. We have $$\{\ell_z,\{\ell_x,\ell_y\}\}=- 2\sum_{p\in x\cap y,q\in z\cap x}\frac{\cosh\big(\ell_{pq}-\frac{a}{2}\big)}{S_a}\sin\theta_p\sin\theta_q
 + 2\sum_{r\in x\cap y,s\in z\cap y}\frac{\cosh\big(\ell_{rs}-\frac{b}{2}\big)}{S_b}\sin\theta_r\sin\theta_s.$$
\end{theorem}
\begin{proof}
 Since $\tr$ is a Lie algebra homomorphism, 
 \begin{align*}  & \{\tr_z, \{\tr_x, \tr_y\}\}\\
        =\ & \{2\varepsilon_z   \cosh\frac{\ell_z}{2},\{2\varepsilon_x\cosh\frac{\ell_x}{2}, 2\varepsilon_y\cosh\frac{\ell_y}{2}\} \}\\
       =\ & \{2\varepsilon_z \cosh\frac{\ell_z}{2}, \varepsilon_x\varepsilon_y\sinh{\frac{\ell_x}{2}}\sinh{\frac{\ell_y}{2}}\{\ell_x,\ell_y\} \}\\
       =\ &\varepsilon_x \varepsilon_y \varepsilon_z \sinh{\frac{\ell_z}{2}}\{\ell_z, \sinh{\frac{\ell_x}{2}}\sinh{\frac{\ell_y}{2}}\{\ell_x,\ell_y\} \}\\
       =\ & \varepsilon_x \varepsilon_y \varepsilon_z \sinh{\frac{\ell_z}{2}} \bigg[ \{\ell_z, \sinh\frac{\ell_x}{2}\}\sinh{\frac{\ell_y}{2}}\{\ell_x,\ell_y\}+\sinh{\frac{\ell_x}{2}}\{\ell_z,\sinh\frac{\ell_y}{2}\}\{\ell_x,\ell_y\} \\ 
       & \qquad \qquad \qquad + \sinh{\frac{\ell_x}{2}}\sinh{\frac{\ell_y}{2}}\{\ell_z,\{\ell_x,\ell_y\}\}\big]\\
       =\ & \varepsilon_x \varepsilon_y \varepsilon_z \sinh{\frac{\ell_z}{2}} \Big [\ \frac{1}{2}\cosh\frac{\ell_x}{2}\sinh{\frac{\ell_y}{2}}\{\ell_x,\ell_y\}\{\ell_z, \ell_x\} +\frac{1}{2}\sinh\frac{\ell_x}{2}\cosh\frac{\ell_y}{2}\{\ell_x,\ell_y\}\{\ell_z,\ell_y\} 
       \\ 
    & \qquad \qquad \qquad + \sinh\frac{\ell_x}{2}\sinh\frac{\ell_y}{2}\{\ell_z,\{\ell_x,\ell_y\}\} \Big]\\
    \end{align*}
    Substituting the values of $\{\ell_x,\ell_y\}$ and $\{\ell_x,\ell_z\}$ and $\{\ell_y,\ell_z\}$ using  Theorem \ref{thm:cosine form}) we obtain

    \begin{align*}
  =\ & \varepsilon_x \epsilon_y \varepsilon_z \Big[-2 C_aS_bS_c\sum_{p\in x\cap y, q\in x\cap z}\cos\theta_p\cos\theta_q\\
        &\ \quad\quad -{2}S_aC_bS_c\sum_{p\in x\cap y, q\in y\cap z} \cos\theta_p\cos\theta_q+ S_aS_bS_c\{\ell_z,\{\ell_x,\ell_y\}\}\Big].
\end{align*}  
     Comparing the equation with the value of $\{\tr_z,\{\tr_x,\tr_y\}\}$ in Theorem \ref{thm:sine trace} and substituting $\varepsilon_x=\varepsilon_A=1,$ $\varepsilon_y=\varepsilon_B=1$ and $\varepsilon_z=\varepsilon_C=1$, we obtain  
     $$\{\ell_z,\{\ell_x,\ell_y\}\}=- 2\sum_{p\in x\cap y,q\in z\cap x}\frac{\cosh\big(\ell_{pq}-\frac{a}{2}\big)}{S_a}\sin\theta_p\sin\theta_q
 + 2\sum_{r\in x\cap y,s\in z\cap y}\frac{\cosh\big(\ell_{rs}-\frac{b}{2}\big)}{S_b}\sin\theta_r\sin\theta_s.$$ 
     \end{proof}    
\begin{remark}
    To obtain the standard form of Wolpert's sine formula, observe the following.  Suppose $l_1=\ell_{pq}$ and $l_2=\ell_x-l_1$, then  $l_1+l_2=\ell_x$  and $ \displaystyle\frac{\cosh\left(l_1-\frac{\ell_x}{2}\right)}
{\sinh\frac{\ell_y}{2}}=\frac{e^{l_1}+e^{l_2}}{(e^{\ell_x}-1)}.$ Similarly if we consider
$ m_1=\ell_{rs}$  and $m_2=\ell_y-m_1$, then $m_1+m_2=\ell_y$ and $\displaystyle\frac{\cosh\left(m_1-\frac {\ell_y}{2}\right)}{\sinh\frac{\ell_y}{2}}=\frac{e^{m_1}+e^{m_2}}{(e^{\ell_y}-1)}.$ Therefore by Theorem \ref{thm:sine form},
\begin{align*}
    \{\ell_z,\{\ell_x,\ell_y\}\}=  2\sum_{r\in x\cap y,s\in z\cap y}\frac{e^{m_1}+e^{m_2}}{(e^{\ell_y}-1)}\sin\theta_r\sin\theta_s- 2\sum_{p\in x\cap y,q\in z\cap x}\frac{e^{l_1}+e^{l_2}}{(e^{\ell_x}-1)}\sin\theta_p\sin\theta_q.
\end{align*}
\end{remark}
\subsection{Convexity of length functions}
As a consequence of our results, we prove that geodesic length functions are convex along Fenchel-Nielsen twist deformations about simple closed curves. Observe that although the Atiyah-Bott-Goldman symplectic form agrees with the Weil-Petersson symplectic form only up to a nonzero scalar, the sign of this scalar does not affect iterated Poisson brackets involving three functions.

\begin{proposition}\label{rem: sim geo}  Let $x$ be a simple geodesic, then
\begin{align*}
   \{\tr_x, \{\tr_x, \tr_y\}\} = -2S_a^2C_b\left(\sum_{p\in x\cap y}\cos\theta_p\right)^2+ 2S_a^2S_b\sum_{p,q\in x\cap y}\frac{\cosh\big(\ell_{pq}-\frac{b}{2}\big)}{S_b}\sin\theta_p\sin\theta_q.
\end{align*}
\end{proposition}
\begin{proof}
  Since $x$ is simple, we have $x\cap x=\varnothing$. Replacing $z$ by $x$ in the proof of Theorem \ref{thm:sine trace}, we can deduce the following:

  \begin{align*}
   \{\tr_x, \{\tr_x, \tr_y\}\} = -2S_a^2C_b\sum_{p,q\in x\cap y}\cos\theta_p\cos\theta_q+ 2S_a^2S_b\sum_{p,q\in x\cap y}\frac{\cosh\big(\ell_{pq}-\frac{b}{2}\big)}{S_b}\sin\theta_p\sin\theta_q.
\end{align*} 
The binomial formula completes the proof.
\end{proof}

\begin{theorem}\label{prop: simple tr}
    Let $x$ be a simple geodesic and $y$ be any geodesic such that $x\cap y\neq \emptyset$. Then 
    $\{\ell_x,\{\ell_x,\ell_y\}\}> 0$.
\end{theorem}
\begin{proof}
    By Theorem \ref{thm:sine form} and Proposition \ref{prop: simple tr}, we obtain
    \begin{align*}
         \{\ell_x,\{\ell_x,\ell_y\}\}= 2S_a^2S_b\sum_{p,q\in x\cap y}\frac{\cosh\big(\ell_{pq}-\frac{b}{2}\big)}{S_b}\sin\theta_p\sin\theta_q.
    \end{align*}
   Since, $\cosh x\geq 1$, $ 0<\theta_p<\pi$ and $0<\theta_q<\pi$ the statement follows. 
\end{proof}

\bibliographystyle{amsplain}

\bibliography{gla.bib}
\end{document}